\documentclass[12pt]{article}
\usepackage[ansinew]{inputenc}

\usepackage{latexsym,amsmath,amsthm,amssymb,amsfonts,amscd}
\usepackage{pstricks,pst-node,pst-plot}

\usepackage{graphicx}
\usepackage{color}

\usepackage{graphics}

\date{\today}

\newcommand{\indic}[1]{\mathbf{1}_{#1}}
\newcommand{\indica}[1]{\mathbf{1}_{\{#1\}}}

\newtheorem{theorem}{Theorem}

\newtheorem{lemma}[theorem]{Lemma}
\newtheorem{coro}[theorem]{Corollary}
\newtheorem{proposition}[theorem]{Proposition}
\theoremstyle{definition}
\newtheorem{defn}[theorem]{Definition}
\newtheorem{remark}[theorem]{Remark}

\newcommand{\eps}{\varepsilon}

\newcommand{\R}{\mathbb{R}}

\newcommand{\bck}{\!\!\!}
\newcommand{\bx}{{\bf x}}
\newcommand{\FF}{\mathcal{F}}
\newcommand{\La}{\Lambda}
\newcommand{\X}{\Xi}
\newcommand{\la}{\lambda}
\newcommand{\iy}{\infty}
\newcommand{\sfrac}[2]{{\textstyle\frac{#1}{#2}}}
\newcommand{\N} {\mathbb{N}}

\newcommand{\sumd}{\sum_{i=1}^\iy x_i^2}

\newcommand{\s}{\bar s}

\begin{document}

\author{Vlada Limic\thanks{This article reports on the research done during a visit
to the Institut Mittag-Leffler (Djursholm, Sweden).}\thanks{Research supported in part by the ANR MAEV grant.}\\
CNRS UMR 6632}

\title{On the speed of coming down from infinity for $\X$-coalescent processes}

\maketitle

\begin{abstract}
\noindent
The $\X$-coalescent processes were 
initially studied by  M\"{o}hle and Sagitov (2001), and
introduced by Schweinsberg (2000) in their full generality.
They arise in the mathematical population genetics as
the complete class of scaling limits for genealogies of Cannings' models. 
The $\X$-coalescents generalize $\Lambda$-coalescents, where now simultaneous 
 multiple collisions of blocks are possible.
The standard version starts with infinitely many blocks at time $0$, and it is said
to come down from infinity if its number of blocks becomes immediately finite, almost surely.
This work builds on the technique introduced recently by Berestycki, Berestycki and Limic (2009),
%to establish an (almost) exhaustive criterion for the above property, in terms of the
%underlying measure $\X$.
%Moreover, 
and exhibits a deterministic ``speed'' function -- an almost sure 
small time asymptotic to the number of blocks process,
for a large class of $\X$-coalescents that come down from infinity.
\end{abstract}

\noindent {\em AMS 2000 Subject Classification.}
60J25, 60F99, 92D25

\noindent {\em Key words and phrases.}
Exchangeable coalescents, small-time asymptotics, coming down from infinity,
martingale technique

\noindent {\em Running Head:}
On the speed of CDI for $\X$-coalescents

\clearpage

\section{Introduction}
Kingman's coalescent \cite{king82,king82b} is one of the 
central models of mathematical population genetics.
From the theoretical perspective, its importance is 
linked to the duality with 
the Fisher-Wright diffusion (and more generally with the Fleming-Viot process).
Therefore the Kingman coalescent emerges in the scaling limit of genealogies
of all evolutionary models that are asymptotically linked to Fisher-Wright
diffusions.
From the practical perspective, its elementary nature allows for exact
computations and fast simulation, making it amenable to statistical analysis.

Assume that the original sample has $m$ individuals, labeled $\{1,2,\ldots,m\}$.
One can identify each of the active ancestral
lineages, at any particular time, with a unique equivalence
class of $\{1,2,\ldots,m\}$ that consists of all the individuals 
that descend from this lineage.
In this way, the coalescent event of two ancestral lineages can be perceived as 
the merging event of two equivalence classes.
Ignoring the partition structure information, 
one can now view the coalescent as a {\em block} (rather than equivalence class) merging
process.

Kingman's coalescent corresponds to the dynamics where
each pair of blocks coalesces at rate 1.
Hence, if there are $n$ blocks present in the current configuration, 
the total number of blocks decreases by $1$ at rate ${n\choose 2}$.
Using this observation and elementary properties of
exponential random variables, one can quickly construct the {\em standard version}
of the process,
which ``starts'' from a configuration containing an infinite number of individuals (particles, or blocks)
at time $0$, 
and has the property that its configuration consists of finitely many blocks 
at any positive time.

The fact that in the Kingman coalescent dynamics only pairs of blocks can merge at any given time
makes it less suitable to 
model evolutions of marine populations or
viral populations under strong selection.
In fact, it is believed (and argued to be observed in experiments, see e.g.~\cite{marine}) that 
in such settings the reproduction mechanism allows for a proportion of 
the population to have the same parent (i.e., first generation ancestor).
This translates to having multiple
collisions of
the ancestral lineages in the corresponding coalescent mechanism.
%They also emerge in the fine-scale mapping of disease loci \cite{morrisetall}.

A family of mathematical models with the above property was independently introduced and
studied by Pitman \cite{pit99} and Sagitov \cite{sag99}
under the name {\em $\Lambda$-coalescents} or
{\em coalescents with multiple collisions}.
Almost immediately emerged 
an even more general class of models, named {\em $\X$-coalescents} or 
{\em coalescents with simultaneous multiple collisions} or {\em exchangeable coalescents}.
The Greek letter $\X$ in the name is a reference to 
the {\em driving measure $\X$} (see Sections \ref{S:Xicoal} and \ref{S:Xicoal further} for details).
The $\X$-coalescent processes were 
initially studied by  M\"{o}hle and Sagitov \cite{moehle sagitov}, and
introduced by Schweinsberg \cite{schweinsberg_xi} in their full generality.
In particular, it is shown in \cite{moehle sagitov} that any limit
of genealogies arising from a population genetics model with exchangeable reproduction
mechanism must be a $\X$-coalescent.

The current paper uses the setting and several of the results from \cite{schweinsberg_xi} 
that will be recalled soon.
Formally, under the $\X$-coalescent dynamics,
several families of blocks (with two or more blocks in each family) 
may (and typically do) coalesce simultaneously.
The $\X$-coalescents will be rigorously defined in the next section.

More recently, 
Bertoin and Le Gall \cite{blg1} established a one-to-one correspondence between a class of 
processes called {\em stochastic flows of 
bridges} and $\X$-coalescents, and then constructed in \cite{blg3}
the {\em generalized Fleming-Viot (or $\La$-Fleming-Viot) processes} 
that have $\La$-coalescent processes as duals;
Birkner et al.~\cite{biretal}~recently extended this further by constructing
for each driving measure $\X$ the
{\em $\X$-Fleming-Viot processes}, dual to the corresponding
$\X$-coalescent process;
Durrett and Schweinsberg \cite{dursch} 
showed that
genealogies during {\em selective sweeps}
are well-approximated 
by certain $\X$-coalescents;
and Birkner and Blath \cite{birbla} initiated a statistical study of coalescents with
multiple collisions.

Generalizations of
$\Lambda$-coalescents to spatial (not a mean-field) setting are studied by
Limic and Sturm \cite{ls}, and more recently by Angel et al.~\cite{abl} and 
Barton et al.~\cite{barethveb}.
The reader can find detailed information about these and related
research areas in recent texts by Berestycki \cite{berest brasil} and 
Bertoin \cite{bertoin book}.

Let $N^\X\equiv N:=(N(t),\,t \geq 0)$ be the number of blocks process
corresponding to a particular standard (meaning $\lim_{t\to 0+}N(t)=\iy$) $\X$-coalescent process.
Moreover, suppose that this $\X$-coalescent {\em comes down from infinity}, or equivalently, assume that 
$P(N(t)<\infty,\,\forall t>0)=1$ (see the end of
Section \ref{S:Xicoal further} for a formal discussion).
From the practical perspective,
it seems important to understand the nature of the divergence of $N(t)$
as $t$ decreases to $0$ (see \cite{bbl2} for further discussion and applications).

The main goal of this work is
to exhibit a function $v_\X\equiv v:(0,\iy)\to (0,\iy)$ such that
\begin{equation}
\label{Ealmost surely}
\lim_{t\to 0+}
\frac{N(t)}{v(t)}= 1, \mbox{ almost surely.}
\end{equation}
We call any such $v$ the {\em speed} 
of coming down from infinity (speed of CDI)
for the corresponding $\X$-coalescent.
The exact form of the function $v$ is implicit and somewhat
technical, see Theorems \ref{Tsmalltime reg early} or \ref{Tsmalltime reg} for a precise statement.
Moreover, the speed is obtained under (relatively weak) additional ``regularity'' condition.

The coming down from infinity property was already studied by Schweinsberg \cite{schweinsberg_xi}
in detail.
% zz and the technique presented here could yield similar (although somewhat 
% weaker) conditions for coming down from infinity. 
The speed of coming down from infinity for general $\X$-coalescents has not been previously studied.
In Berestycki et al.~\cite{bbl1} the speed of CDI of
any $\La$-coalescent that comes down from infinity
was found using a martingale-based technique. 
A modification of this technique will be used presently to determine the above $v$,
and the steps in the argument that carry over directly to the current setting will
only be sketched.

In the $\La$-coalescent setting, weaker asymptotic 
results  (than (\ref{Ealmost surely})) on $N^\X/v=N^\La/v$ 
%for the corresponding $v$ 
can be deduced
by an entirely different approach, based on the theory of L\'evy processes and superprocesses. 
This link was initially discovered in \cite{bbs2,bbs1} in the special case of so-called Beta-coalescents,
and recently understood in the context of general $\La$-coalescents in \cite{bbl2}.
It is worthwhile pointing out, that for any ``true'' $\X$-coalescent (meaning that simultaneous
multiple collisions are possible in the dynamics), an approach 
analogous to \cite{bbl2} seems rather difficult to implement (to start with,
the expression
(\ref{D:psi xi}), unlike (\ref{D:psi la}), does not seem to be directly linked with any 
well-known stochastic process). 
Indeed, the martingale technique from \cite{bbl1} has at least three advantages: (i) 
it yields stronger forms of convergence, more precisely, (\ref{Ealmost surely}) and its counterparts in the
$L^p$-sense, for any $p\geq 1$ (cf.~\cite{bbl1} Theorems 1 and 2); 
it yields explicit error estimates needed in the frequency spectrum analysis 
(see \cite{bbl1,bbl2} for details);
and (iii) it extends 
to the $\X$-coalescent setting as will be explained shortly.

It is not surprising that the form of the ``candidate'' speed of CDI for the $\X$-coalescent 
is completely analogous to that for the $\La$-coalescent.
What may be surprising is that there are $\X$-coalescents that come down from infinity but their
candidate speed is identically infinite. And also that there might be coalescents 
that come down from infinity but faster than their finite candidate speed.
(Remark \ref{R cdi lam} in Section \ref{S:Xicoal further}
explains how neither of these can occur under a $\La$-coalescent mechanism.)
The question of whether an asymptotic speed still exists in such cases remains open. 
This discussion will be continued in Section \ref{S:improper joint}.

The rest of the paper is organized as follows: Section \ref{S:prelim} introduces various
processes of interest (including in Section \ref{S:two op} the novel 
``color-reduction'' and ``color-joining'' constructions
that might be of independent interest)
and presents a ``preview'' of the main result as Theorem \ref{Tsmalltime reg early}.
Section \ref{S:results} contains a ``matured'' 
statement of the main result (Theorem \ref{Tsmalltime reg}), followed by a
discussion of some of its immediate consequences, and of the significance of 
a certain ``regularity hypothesis'', while Section \ref{S:martingale} is devoted to 
the proof of Theorem \ref{Tsmalltime reg}.

\section{Definitions and preliminaries}
\label{S:prelim}
\subsection{Notation}
\label{S:notation} 
%We recall some standard notation, and  introduce
%additional notation to simplify the exposition.
In this section we recall some standard notation, as well as gather
less standard notation that will be frequently used.

Denote the set of real 
%(resp.~rational) 
numbers by $\R$
%(resp.~$\Q$) 
and set $\R_+ = (0,\iy)$.
For $a,b\in \R$, denote by $a \wedge b$ (resp.~$a \vee b$) the minimum (resp.~maximum)
of the two numbers.
Let
\begin{equation}
\label{EDel}
\Delta := \{(x_1,x_2,\ldots): x_1\geq x_2 \geq \ldots \geq 0,\, \sum_i x_i \leq 1\},
\end{equation}
be the infinite unit simplex.
If $\bx=(x_1,x_2,\ldots) \in \Delta$ and $c\in \R$, let 
\[
c\, \bx = (c x_1,c x_2,\ldots).
\]
Denote by $0$ the {\em zero} $(0,0\ldots,)$ in $\Delta$.

Let $\N:=\{1,2,\ldots\}$, and $\mathcal{P}$ be the set of partitions of $\N$.
Furthermore, for $n\in \N$ denote by
$\mathcal{P}_n$ the set of partitions of $[n]:=\{1,\ldots, n\}$.

If $f$ is a function, defined in a left-neighborhood $(s-\eps,s)$ of a point $s$,
 denote by $f(s-)$ the left limit of $f$ at $s$.
Given two functions $f,g:\R_+\to \R_+$, write
$f=O(g)$ if $\limsup f(x)/g(x) <\infty$,
$f=o(g)$ if $\limsup f(x)/g(x) =0$,
 and $f\sim g$ if $\lim f(x)/g(x) =1$.
Furthermore, write $f=\Theta(g)$ if both $f=O(g)$ and $g=O(f)$.
The point at which the limits are taken
is determined from the context.

%If $X$ and $Y$ are two random objects we write $X\overset{d}=Y$ to indicate 
%their
%equivalence in distribution.
%As usual, convergence in distribution will be denoted by
% $\Rightarrow$ symbol.

If $\FF=(\FF_t,t\ge 0)$ is a filtration, and $T$ is a stopping time
relative to $\FF$, denote by $\FF_T$ the standard filtration generated by $T$,
see for example \cite{durrett}, page 389.

For $\nu$ a finite or $\sigma$-finite measure on $\Delta$ or on $[0,1]$, denote the
support of $\nu$ by ${\rm supp} (\nu)$.

\subsection{$\X$-coalescents}
\label{S:Xicoal}
Let $\X$ be a finite measure on $\Delta$, and write 
\[
\X=\X_0 + a \delta_0,
\]
where $a\geq 0$ and $\X_0((0,0,\ldots))=0$.
As noted in \cite{schweinsberg_xi}, we may assume without loss of generality that $\X$
is a probability measure.
The {\em $\X$-coalescent} driven by the above $\X$
is a Markov process $(\Pi_t,t\ge 0)$ with values
in ${\cal P}$ (the set of partitions of $\N$),
characterized in the following way. If $n \in \N$,
then the restriction
$(\Pi^{(n)}_t,t\ge 0)$
of $(\Pi_t,t\ge 0)$ to
$[n]$ is a Markov chain, taking values in $\mathcal{P}_n$,
such that
while $\Pi^{(n)}_t$ consists of $b$ blocks,
any given $k_1$-tuple, $k_2$-tuple,$\ldots$, and
$k_r$-tuple 
of its blocks (here $\sum_{i=1}^r k_i \leq b$ and $k_i\geq 2$, $i=1,\ldots,r$)
merge simultaneously (each forming one new block) at rate
\begin{equation*}
\label{rate coal xi}
\lambda_{b;k_1,\ldots,k_r;\s}=\int_{\Delta} \frac{\sum_{l=0}^{\s} \sum_{i_1,\ldots,i_{r+l}}
{\s \choose l} x_{i_1}^{k_1}\cdots x_{i_r}^{k_r} x_{i_{r+1}}\cdots x_{i_{r+l}} (1 - \sum_{i=1}^\iy x_i)^{\s-l} }{\sumd}\,\X(d\bx),
\end{equation*}
where $\s:=b-\sum_{i=1}^r k_i$ is the number of blocks that do not
participate in the merger event, and where the sum $\sum_{i_1,\ldots,i_{r+l}}$ in the above summation 
stands for the infinite sum
$\sum_{i_1=1}^\iy \sum_{i_2=1,i_2\neq i_1}^\iy \ldots \sum_{i_{r+l}=1,i_{r+l}\not \in\{i_1,\ldots, i_{r+l-1}\}}^\iy $ over $r+l$ different indices.
It is easy to verify that each such coalescent process has the same rate of pairwise merging
\begin{equation}
\label{E scaling}
\la_{2;2;0}=\X(\Delta)=1.
\end{equation}

\subsection{Preview of the small-time asymptotics}
\label{S:main preview}
One can now state the central result of this paper.
Given a probability measure $\X$ as above, for each $t>0$
denote by $N^\X(t)$ the number of blocks  at time $t$ in the 
corresponding (standard) $\X$-coalescent process.
Define
\[
\psi_{\X}(q) := \int_{\Delta}\frac{\sum_{i=1}^\iy(e^{-qx_i}-1+qx_i)}{\sumd}\,\X(d\bx),  \ q \geq 0,
\]
and 
\[
v_{\X} (t) : = \inf\left\{s>0: \int_s^{\infty} \frac1{\psi_\X(q)}\,dq <t\right\}, \ t>0.
\]
\begin{theorem} \label{Tsmalltime reg early}
If both 
\[
\X(\{\bx \in \Delta: \sum_{i=1}^n x_i = 1 \mbox{ for some finite }n\})=0
\]
and
\[
\int_\Delta \frac{(\sum_{i=1}^\iy x_i)^2}{\sumd} \,\X(d\bx) < \iy,
\] 
then
\begin{equation*}
%\label{nblocks}
\lim_{t\to 0+}\frac{N^\X(t)}{v_\X(t)} = 1, \mbox{ almost surely},
\end{equation*}
where $\iy/\iy \equiv 1$.
In particular, under the above assumptions,
the quantity $v_\X(t)$ is finite for (one and then for) all $t>0$, if and only if
the $\X$-coalescent comes down from infinity.
\end{theorem}

Most of the sequel is devoted to explaining
the above implicit definition of the speed $v_\X$, as well as
the significance of the two hypotheses in Theorem \ref{Tsmalltime reg early}.

As already mentioned, Theorem \ref{Tsmalltime reg early} is
restated as Theorem \ref{Tsmalltime reg} in Section
\ref{S:results}, which is proved in Section \ref{S:martingale}.
The additional condition $\X(\{0\})=0$ in Theorem \ref{Tsmalltime reg} is not really 
restrictive, since the case where $\X(\{0\})>0$ is already well-understood 
(cf.~Remark \ref{R:atom zero} below).

\subsection{Basic properties of $\X$-coalescents}
\label{S:Xicoal further}
Recall the setting and notation of Section \ref{S:Xicoal}.

If ${\rm supp}(\X)\subset\{(x,0,0,\ldots): x\in [0,1]\}$,
the resulting $\X$-coalescent is usually called
the {\em $\Lambda$-coalescent},
where $\La$ is 
%the probability measure on $[0,1]$ 
specified by 
\begin{equation}
\label{ELa from Xi}
\La(dx):=\X(d(x,0,\ldots)).
\end{equation}
The transition mechanism simplifies as follows: 
whenever $\Pi^{(n)}_t$ consists of $b$ blocks, the
rate at which any given $k$-tuple of its blocks merges into a single block equals
\begin{equation}
\label{rate coal la}
\lambda_{b,k}=\int_{[0,1]}x^{k-2}(1-x)^{b-k}\Lambda(dx).
\end{equation}
Note that mergers of several blocks into one are still possible here, but
multiple mergers cannot occur simultaneously. 

We recall several
properties of the $\X$-coalescents carefully established in \cite{schweinsberg_xi}, 
the reader is referred to this article for details.
The $\X$-coalescents can be constructed via
a Poisson point process in
the following way. 
Assume that $\X=\X_0$, or equivalently
that $\X$ does not have an atom at $0$ (see also Remark \ref{R:atom zero} below).
Let
\begin{equation}
\label{DPPPpi}
\pi(\cdot) = \sum_{k \in \N } \delta_{t_k,\bx_k}(\cdot)
\end{equation}
be a Poisson point process on $\R_+ \times \Delta$ with intensity measure
$dt \otimes \X(d\bx)/\sumd$.
Each atom $(t,\bx)$ of $\pi$ influences the evolution of the process $\Pi$ as
follows:
to each block of $\Pi(t-)$ assign a random ``color'' in an i.i.d.~fashion
(also independently of the past)
where the colors take values in $\N\cup (0,1)$
and their common distribution $P_\bx$ is specified by
\begin{equation}
\label{Dcoloring}
P_\bx(\{i\})=x_i, \ i\geq 1 \ \mbox{ and } P_\bx(du)=(1-\sum_{i=1}^\iy x_i)\,du,\ u\in (0,1);
\end{equation}
given the colors, 
merge immediately and simultaneously all the blocks of equal color into a single block
(note that this can happen only for integral colors),
while leaving the blocks of unique color unchanged.

Note that in order to make this construction rigorous,
%definition of the process
%$(\Pi_t,t\ge 0)$,
one should first consider the restrictions
$(\Pi^{(n)}(t),t\ge 0)$, since the measure $\X(d\bx)/\sumd$ may have 
(and typically will have in the cases of interest) infinite total mass.
Given a fixed time $s>0$, a small $\eps>0$ and any $n\in \N$,
it is straightforward to run the above procedure 
using only the finite number of atoms of $\pi$ that are contained in
$[0,s] \times \{\bx\in \Delta : \sumd > \eps\}$.
Denote the resulting process by $(\widetilde{\Pi}^{(n),\eps}(t),t\in [0,s])$.
The following observation is essential:
%for any given 
%on any given time interval $[0,t]$, 
the atoms of $\pi$ contained in the ``complement'' $[0,s] \times \{\bx\in \Delta : \sumd \leq \eps\}$, 
together with the coloring procedure,
influence further the state of $(\widetilde{\Pi}^{(n),\eps}(t),t\in [0,s])$ 
on the event $A_{\eps;s,n}$ (by causing additional mergers during $[0,s]$),
where 
\[
P(A_{\eps;s,n})\leq 
1- \exp\left(\int_0^s dt \int_{\Delta \cap \{\sumd \leq \eps\}} 
\frac{R_\bx(n)}{\sumd}
\,\X(d\bx)\right),\]
and where the right-hand-side goes to $0$ as $\eps\to 0$.
Indeed, $R_\bx(n)$ is the probability that under $P_\bx$ (assuming $n$ blocks present at time $t-$) 
at least two of the blocks are colored by the same color.
For the benefit of the reader we include the exact expression for this probability:
\begin{eqnarray*}
R_\bx(n)&=& 1- (1-\sum_{i=1}^\iy x_i)^n - n\sum_{i=1}^\iy x_{i} (1-x_{i})^{n-1} \\
&-&{n\choose 2}\sum_{i_1=1}^\iy \sum_{i_2=1,i_2\neq i_1}^\iy x_{i_1} x_{i_2} (1-x_{i_1}-x_{i_2})^{n-2} - \ldots\\\
&-& n \sum_{i_1=1}^\iy \ldots \sum_{i_{n-1}=1,i_{n-1}\not \in \{i_1,i_2,\ldots,i_{n-2}\}}^\iy
\prod_{\ell =1}^{n-1} x_{i_\ell} (1-\sum_{\ell=1}^{n-1}x_{i_\ell}) \\ 
&-& 
\sum_{i_1=1}^\iy \ldots \sum_{i_n=1,i_n\not \in \{i_1,i_2,\ldots,i_{n-1}\}}^\iy \prod_{\ell =1}^n x_{i_\ell}.
\end{eqnarray*}
Note that $R_\bx(n)\leq {n\choose 2}\sumd$, so that 
\[
\int_0^s dt \int_{\Delta \cap \{\sumd \leq \eps\}} 
\frac{R_\bx(n)}{\sumd}
\,\X(d\bx) \leq s {n\choose 2} \int_{\Delta \cap \{\sumd \leq \eps\}} \,\X(d\bx) \to 0, \mbox{ as } \eps \to 0.
\]
In this way
one obtains a coupling (that is, a simultaneous construction on a single probability space)
of the family of processes $(\widetilde{\Pi}^{(n),\eps}(t),\, t\in [0,s])$, as $\eps \in (0,1)$,
and can define ${\Pi}^{(n)}$ as the limit
${\Pi}^{(n)}:= \lim_{\eps\to 0} \widetilde{\Pi}^{(n),\eps}$ on [0,s].
Moreover, the above construction is amenable to appending particles/blocks to the initial configuration,
hence it yields a coupling of   
\begin{equation}
\label{Ecoupl eps n}
(\widetilde{\Pi}^{(n),\eps}(t),\, t\in [0,s]),\mbox{  as }\eps \in (0,1)\mbox{ and }n\in \N.
\end{equation}
An interested reader is invited to check (or see~\cite{schweinsberg_xi}) that the limit 
\begin{equation}
\label{Exi constr}
{\Pi}:=\lim_{n\to \iy} {\Pi}^{(n)}
\end{equation}
is a well-defined realization of the $\X$-coalescent, corresponding
to the measure $\X$. 
We will denote its law simply by $P$ (rather than by $P_\X$). 

Consider the above $\X$-coalescent process $\Pi$.
Let $E=\{N(t)=\infty$ for all $t\geq 0\}$, and 
$F=\{N(t)<\infty$ for all $t>0\}$.
\begin{defn}
We say that a $\X$-coalescent {\em comes down from infinity} if $P(F)=1$.
\end{defn}
Let
\begin{equation}
\label{EDelfin}
\Delta_f := \{\bx \in \Delta: \sum_{i=1}^n x_i = 1 \mbox{ for some finite }n\}.
\end{equation}
Lemma 31 \cite{schweinsberg_xi} extends Proposition 23 of Pitman \cite{pit99} to the
$\X$-coalescent setting.
It says that provided $\X(\Delta_f)=0$, there are two possibilities for the 
evolution of $N$:
either $P(E)=1$ or $P(F)=1$.
\begin{remark}
\label{R:atom zero}
A careful reader will note that the above Poisson point process (PPP) construction assumed
$\X(\{0\})=0$.
It is possible to enrich it with extra pairwise mergers if $\X(\{0\})>0$, see \cite{schweinsberg_xi} for details.
For the purposes of the current study this does not seem to be necessary. %for the following reason.
Indeed, by the argument of \cite{bbl1} Section 4.2, one can easily see 
that if $\X((0,0,\ldots,))=a>0$, then the corresponding $\X$-coalescent comes down from infinity, and moreover its speed of CDI is 
determined by $a$. 
More precisely, such a $\X$-coalescent comes down faster than the $\La$-coalescent 
$(\Pi_a(s),s\geq 0)$ corresponding to 
$\La(dx)=a \delta_0(dx)$
(note that $\Pi_a$ is just a time-changed Kingman coalescent), and slower than 
 $(\Pi_a((1+\eps)s),s\geq 0)$, for any $\eps>0$.
\end{remark}
In the rest of this paper we will assume that 
$\X(\{0\})=0$, or equivalently, that $\X=\X_0$.
\begin{remark}
\label{R pEpF}
The condition $\X(\Delta_f)=0$ is similar (but not completely analogous, see next remark)
to the condition $\Lambda(\{1\})=0$ for $\La$-coalescents.
It is not difficult to construct a $\X$-coalescent, 
such that $\X(\Delta_f)>0$ and $P(E)=P(F)=0$. Take some probability measure $\X'$ on $\Delta$ 
such that the corresponding $\X$-coalescent does
not come from infinity, and define 
$\X=(1-a)\X'+ a \nu$, for some $a\in (0,1)$ and some probability measure $\nu$ on $\Delta_f$
(for example $\nu(d\bx)=\delta_{(1/2,1/2,0,\ldots)}(d\bx)$).
Then its block counting process stays infinite for all times strictly smaller than $T_*$,
and it is finite for all times larger than or equal to $T_*$, where
\[
T_*:=\inf\{s: \pi(\{s\}\times \Delta_f)>0\} 
\]
has exponential (rate $a$) distribution, hence is strictly positive with probability $1$.
%See Remark \ref{R delfin} for further discussion.
\end{remark}
\begin{remark}
\label{R delfin}
It may be surprising that there are measures $\X$ satisfying $\X(\Delta_f)=1$,
and such that the corresponding $\X$-coalescent comes down from infinity.
Note that there is no analogy in the setting of $\La$-coalescents, since if $\La(\{1\})=1$
(and therefore $\La([0,1))=0$), the only such ``$\La$-coalescent'' will 
contain a single block for all times.

%As usual, assume that $\X(\{0\})=0$, and recall (\ref{EDelfin}).
It was already observed by Schweinsberg \cite{schweinsberg_xi} Section 5.5 that if the quantity
\[
\int_{\Delta_f} \frac{1}{\sumd}\,\X(d\bx) 
\]
is infinite, the corresponding $\X$-coalescent comes down from infinity.
Moreover, if the above quantity is positive and finite, the corresponding 
$\X$-coalescent comes down from infinity if and only if the $\X$-coalescent
corresponding to $\X'(d\bx)=\X(d\bx)\indica{\bx \in \Delta\setminus \Delta_f}$ comes down from
infinity, and the speed of CDI is determined by $\X'$. 
This type of coalescent was already mentioned in Remark \ref{R pEpF}. The reader should note that 
if such a $\X$-coalescent does not come down from infinity, then $P(E)=P(F)=0$.
\end{remark}
Henceforth we will mostly assume that 
$\X(\Delta_f)=0$.

\subsection{Coming down from infinity revisited}
\label{S:CDIrev}
In this section we assume that $\X(\Delta_f)=0$, as well as $\X(\{0\})=0$.
A sufficient condition for a $\X$-coalescent to
come down from infinity was given by Schweinsberg
\cite{schweinsberg_xi}.  
For $k_i$, $i=1,\ldots,r$ such that 
$k_i\geq 2$ and $\s:=b-\sum_{i=1}^r k_i \geq 0$,
define $N(b;k_1,\ldots,k_r;\s)$ to be the number of different simultaneous choices of a
$k_1$-tuple, a $k_2$-tuple,$\ldots$ and a $k_r$-tuple from a set of $b$ elements.
The exact expression for $N(b;k_1,\ldots,k_r;\s)$ is not difficult to find
(also given in \cite{schweinsberg_xi} display (3)),
but is not important for the rest of the current analysis.
Let 
\[
\gamma_b := \sum_{r=1}^{\lfloor b/2 \rfloor}
\sum_{\{k_1,\ldots,k_r\}} (b-r-\s)
N(b;k_1,\ldots,k_r;\s) \la_{b;k_1,\ldots,k_r;\s}
\]
be the total rate of decrease in the number of blocks for the $\X$-coalescent, when the
current configuration has precisely $b$ blocks.

Given a configuration consisting of $b$ blocks and an $\bx \in \Delta$,
consider the coloring procedure (\ref{Dcoloring}), and define
\begin{equation}
\label{EYs}
Y_\ell^{(b)}:=\sum_{j=1}^b \indica{\text{$i$th block has color $\ell$}}, \ \ell \in \N,
\end{equation}
so that $Y_\ell^{(b)}$ has Binomial($b,x_\ell$) distribution.
Due to the PPP construction of the previous subsection, we then have
\begin{eqnarray}
\gamma_b 
\bck&=&\bck \int_{\Delta} \frac{\sum_{\ell=1}^\infty E(Y_\ell^{(b)} -\indica{Y_\ell^{(b)} >0})}{\sumd} \,\X(d\bx)\nonumber\\
\bck&=&\bck \int_{\Delta} \frac{\sum_{\ell=1}^\infty(bx_\ell-1 +(1-x_\ell)^b)}{\sumd} \,\X(d\bx)\label{Egamma}
\end{eqnarray}

Proposition 32 in \cite{schweinsberg_xi}  says that the $\X$-coalescent comes down from infinity if 
\begin{equation}
\label{E sch cond}
\sum_{b=2}^{\infty} \gamma_b^{-1} <\infty.
\end{equation}
Let $\eps \in (0,1)$ be fixed. 
Recall (\ref{EDel}) and define
\begin{equation}
\label{EDeleps}
\Delta^\eps := \{\bx \in \Delta: \, \sum_i x_i \leq 1-\eps\}.
\end{equation}
Proposition 33 in \cite{schweinsberg_xi} says that (\ref{E sch cond}) is necessary for coming down from
infinity if also
\[
\int_{\Delta\setminus \Delta^\eps} \frac{1}{\sumd}\, \X(d\bx) < \iy, \ \mbox{ for some } \eps>0.
\]
Moreover \cite{schweinsberg_xi} provides an example of a $\X$-coalescent that comes down from
infinity, but does not satisfy
(\ref{E sch cond}). More details are given in Section \ref{S:improper joint}.

The CDI property for $\La$-coalescents is, in comparison, completely understood.
Define
\begin{equation}
\label{D:psi la}
\psi_{\La}(q):= \int_{[0,1]}\frac{(e^{-qx}-1+qx)}{x^2}\La(dx),
\end{equation}
and note that $\gamma_b$ simplifies to
$ \sum_{k=2}^b (k-1){b \choose k} \lambda_{b,k}$, with $\la_{b,k}$ as in (\ref{rate coal la}).
The original sharp criteria is due to Schweinsberg \cite{sch1}:
a particular $\La$-coalescent comes down from infinity if and only if
(\ref{E sch cond}) holds.
Bertoin and Le Gall \cite{blg3} observed that 
\begin{equation}
\label{E blg obs}
\gamma_b= \Theta(\psi_\La(b)),
\end{equation}
 and that therefore the CDI happens if and only if
\begin{equation}
\label{E blg cond}
 \int_a^{\infty} \frac{dq}{\psi_\La(q)} <\infty,
\end{equation}
for some (and then automatically for all) $a >0$.
\begin{remark}
\label{R cdi lam}
A variation of the argument from Berestycki et al.~\cite{bbl1} 
provides an independent (probabilistic) 
derivation of the last claim.
More precisely,
let $N^n(t)=\# \Pi^{(n)}(t)$, $t\geq 0$ and let
$v^n$ be the unique solution of the following Cauchy problem
\[
v'(t)=-\psi(v(t)),\ \ v(0)=n.
\]
Use the argument of \cite{bbl1} Theorem 1 (or see 
Part I in Section \ref{S:smalltime} for analogous
argument in the $\X$-coalescent setting) to find $n_0<\iy$, $\alpha\in (0,1/2)$,
and $C<\infty$ such that
\[
\bigcap_{n\geq n_0} \left\{\sup_{t\in [0,s]}\left|\frac{N^n(t)}{v^n(t)} -1\right| \leq Cs^\alpha\right\}
\]
happens with overwhelming (positive would suffice) probability, uniformly in small $s$.
Finally, note that $v^n$ satisfies the identity
\begin{equation}
\label{Evns}
\int_{v^n(s)}^n \frac{dq}{\psi(q)}=s, \ s\geq 0,
%\mbox{ whenever } s \leq \int_0^n \frac{dq}{\psi(q)},
\end{equation}
therefore $\lim_n v^n(s)<\iy$ if and only if (\ref{E blg cond}) holds.
\end{remark}
\noindent
Moreover, it was shown in \cite{bbl1} that under condition (\ref{E blg cond}), 
a speed $t\mapsto v(t)$ of CDI is specified by
\begin{equation}
\label{Ev la}
\int_{v(t)}^\infty \frac{dq}{\psi_\La(q)} = t, \ t\geq 0.
\end{equation}

Consider again the general $\X$-coalescent setting.
In analogy to (\ref{D:psi la}), define
\begin{equation}
\label{D:psi xi}
\psi(q)\equiv \psi_{\X}(q):= \int_{\Delta}\frac{\sum_{i=1}^\iy(e^{-qx_i}-1+qx_i)}{\sumd}\,\X(d\bx).
\end{equation}
Note that the above integral converges since 
\begin{equation}
\label{Ecalcfact}
e^{-z}-1+z\leq z^2/2, \mbox{ for all }z\geq 0,
\end{equation}
and so in particular $\psi_\X(q)\leq q^2/2$, for any probability measure $\X$ on $\Delta$.
It is easy to check that $q\mapsto \psi_\X(q)$ is an infinitely differentiable, strictly increasing, and convex
function on $\R_+$, as well as that $\psi_\X(q)\sim q^2/2$ as $q\to 0$. 
Therefore, if
\begin{equation}
\label{E me cond}
\int_a^\infty \frac{dq}{\psi(q)}<\iy
\end{equation}
for some $a>0$, the same will be true for all $a>0$,
and irrespectively of that $\int_0^a dq/\psi(q)=\iy$, for any $a>0$.
\begin{lemma}
\label{Lincreasing}
The function $q\mapsto \psi(q)/q$ is strictly increasing.
\end{lemma}

\noindent
The proof (straightforward and left to the reader) is analogous to that for \cite{bbl1} Lemma 9. 
\begin{lemma}
The conditions (\ref{E sch cond}) and (\ref{E me cond}) are equivalent.
\end{lemma}
\begin{proof}
It suffices to show the order of magnitude equivalence (\ref{E blg obs}) in the current setting.
Use expression (\ref{Egamma}) for $\gamma_b$. 
Note that if $x\in [0,1]$ and $b\geq 1$, then
$e^{-bx}\geq (1-x)^b$, in fact for $x\in [0,1)$
\[
e^{-bx} - (1-x)^b  = e^{-bx} \left(1- \exp\left\{-b\left(\frac{x^2}{2}+\frac{x^3}{3}+\ldots\right)\right\}\right).
\]
If $x\geq 1/4$ and $b\geq 16$ it is clearly true that $e^{-bx} - (1-x)^b \leq b x^2 $.
For $x\leq 1/4$ we have 
$\sum_{j\geq 2}x^j \leq 4 x^2/3$, and so %for any $b\geq 1$ we have
\begin{equation*}
1- \exp \left\{-b\left(\frac{x^2}{2}+\frac{x^3}{3}+\ldots
\right)\right\} 
%\leq 1-\exp \left\{-\frac{b}{2} \left(x^2 + x^3 + \ldots \right)\right\} 
 \le 1 - \exp \left(-\frac{2}{3} bx^2\right) 
 \le \frac{2}{3} bx^2.
\end{equation*}
We conclude that
\begin{equation}
\label{E easy bound}
0\leq e^{-bx} - (1-x)^b \leq b x^2, \mbox{ for all }b\geq 16 \mbox{ and }x \in [0,1].
\end{equation}

\noindent
As a consequence,
\begin{equation}
\label{E exp inst pow}
\sum_{\ell=1}^\iy E(Y_\ell^{(b)} -\indica{Y_\ell^{(b)}>0})= \sum_{\ell=1}^\iy (bx_\ell -1 + e^{-bx_\ell}) + b\,O(\sumd),
\end{equation}
where $O(\sumd)\in [-\sumd,0]$.
By integrating over $\X(d\bx)/\sumd$, we get
\[
\gamma_b = \psi(b) + O(b), \mbox{ for some } O(b) \in [-b,0],
\]
implying $\gamma_b=O(\psi(b))$.
It is easy to check directly from (\ref{Egamma}) 
that $\gamma_{b+1}-\gamma_b\geq \gamma_b-\gamma_{b-1}\geq 0$, for any $b\geq 3$
implying $b=O(\gamma_b)$.
Using convexity of $\psi$, we now have that
either $\psi(b)=O(b)$, or
$b=o(\psi(b))$ so that 
\[
\frac{\gamma_b}{\psi(b)}=1 + o(1).
\]
In both cases we have $\gamma_b=\Theta(\psi(b))$.
\end{proof}

Assuming (\ref{E me cond}) (or equivalently, (\ref{E sch cond})),
one can define
\begin{equation}
u_{\X}(t)\equiv u(t):=\int_t^{\infty}\!\frac{dq}{\psi(q)} \in \R_+, \ t>0,
\label{Eu}
\end{equation}
and its c\`adl\`ag inverse
\begin{equation}
\label{Ev} 
v_\X(t)\equiv v(t) :=\inf\left\{s>0: \int_s^{\infty} \frac1{\psi(q)}dq
<t\right\}, \ t>0.
\end{equation}
Call thus defined $v_\X$ the {\em candidate speed}.
In fact, due to the continuity and strict monotonicity of $u$, 
$v_\X$ is again specified by (\ref{Ev la}), with
$\psi_\X$ replacing $\psi_\La$.
If (\ref{Ealmost surely}) holds with $v=v_\X$,
we will sometimes refer to the candidate speed $v_\X$ as the {\em true speed of CDI}.

Note that (\ref{Ev}) makes sense regardless of (\ref{E me cond}), and 
yields $v_\Xi(t) = \infty$, for each $t>0$, if (and only if) (\ref{E me cond}) fails.
We will say that that the $\X$-coalescent ``has an infinite candidate speed'' in this setting.
%$v(t)$ is the unique value $x\in \R_+$ satisfying $\int_x^\iy dq/\psi(q)=t$.

Due to the fact $\psi_\X(q)\leq q^2/2$ (cf.~discussion following (\ref{D:psi xi})) we have
\begin{coro}
\label{C fastest}
If (\ref{E scaling}) holds, then $v_\X(t)\geq 2/t$, for $t> 0$.
\end{coro}
One could try to rephrase the corollary by saying that among all the $\X$-coalescents 
(satisfying (\ref{E scaling})),
the Kingman coalescent is the fastest to come down from infinity at speed $t\mapsto 2/t$
(as is the case in the $\La$-coalescent setting, cf.~\cite{bbl1} Corollary 3).  
However, there are examples of $\X$-coalescents with infinite candidate speed that do come
down from infinity.
Moreover, there are coalescents that come down from infinity, and that have finite candidate speed $v_\X$, 
but the methods of this article break in the attempt of associating $N=N^\X$ and $v_\X$ at small times. 
The existence of a (deterministic) speed, and its relation to the 
function $t\mapsto 2/t$ in these situations, are open problems.
Vaguely speaking, such ``difficult cases'' 
correspond to measures $\X$ for which there exists a set $\Delta_f^*\approx \Delta_f$ such that
\[
\int_{\Delta_f^*} \frac{1}{\sumd}\,\X(d\bx) =\infty. 
\]
For rigorous statements see Section \ref{S:results}.
% and Section \ref{S:improper joint}.

\subsection{Two operations on $\X$-coalescents}
\label{S:two op}
In this section we consider two variations of the PPP construction 
(\ref{Ecoupl eps n})--(\ref{Exi constr}),
each of which gives a probabilistic coupling of the original $\X$-coalescent with a simpler 
$\X$-coalescent.

Given a realization of the 
Poisson point process (\ref{DPPPpi}) and the 
coloring of (\ref{Dcoloring}),
define the {\em $\delta$-reduction} or ({\em $\delta$-color-reduction}) 
$\Pi_\delta^r$ ($r$ stands for ``reduction'') of $\Pi$ to be the 
partition valued process constructed as follows:
%in the same way with additional (and independent) randomness,namely, 
immediately after each coloring step (and before the merging)
run Bernoulli($\delta$) random variable for each block, independently over the blocks and
the rest of the randomness, and for each of the blocks having this new value $1$,
resample its color from the uniform $U[0,1]$ distribution, again independently from everything else.
Note that the ``reduction'' in the name refers to 
reducing the coloring (atom) weights, however this has
the opposite effect on the number of blocks.
Indeed, the above procedure makes some blocks that share (integer) color with others 
in the construction of $\Pi$ become uniquely colored in the construction of $\Pi_\delta^r$.
%By introducing additional randomness (for blocks that already 
%collapsed in $\Pi$ but not in $\Pi^r$), 
With a little extra care, one can obtain a coupling of $\Pi^{(n),\eps}$ (resp.~$\Pi$) 
and its reduction 
$\Pi_\delta^{(n),\eps,r}(t)$ (resp.~$\Pi_\delta^r$), so that
there are fewer blocks contained in $\Pi^{(n),\eps}$ (resp.~$\Pi$)
than in $\Pi_\delta^{(n),\eps,r}$ (resp.~$\Pi_\delta^r$) at all times.

Note that $\Pi_\delta^r$ is also a $\X$-coalescent, and that its driving measure is 
\[
\X_\delta(d\bx) := (1-\delta) \, \X ( d\sfrac{\bx}{1-\delta}) 1_{\{\bx \in \Delta^\delta\}}
\]
Then
$\X_\delta (\Delta)=\X_\delta (\Delta^\delta)= (1-\delta)^2 \X(\Delta)$. 
Let $\psi_\delta^r$ be defined as in (\ref{D:psi xi}), but corresponding to $\Pi_\delta^r$, 
\begin{eqnarray}
\psi_{\delta}^r(q) \equiv \psi_{\X_\delta^r}(q) \bck&=&\bck 
\int_{\Delta^\delta} \frac{\sum_{i=1}^\infty (e^{-qx_i}-1+qx_i) }{\sum_{i=1}^\infty x_i^2}\, (1-\delta)\, \X ( d\sfrac{\bx}{1-\delta}) \nonumber\\
\bck&=&\bck \int_{\Delta}\frac{\sum_{i=1}^\infty (e^{-qx_i(1-\delta)}-1+qx_i(1-\delta)) }
{\sum_{i=1}^\infty x_i^2}\, \X (d\bx). 
\label{Epsi red} 
\end{eqnarray} 
Since
$z \mapsto e^{-z} -1 + z$ is an increasing function  on $[0,\iy)$ 
we have $\psi_\delta^r(q) \leq \psi(q)$.
% in fact, one can easily verify that
%\[
%\psi_\eps^r(q) \leq (1-\eps) \psi(q), \ \forall q \geq 0.
%\]
In fact, (\ref{Epsi red}) states that $\psi_{\delta}^r(q)=\psi ((1-\delta)q)$, $q\geq 0$, hence
\begin{equation*}
\label{E integr equiv red}
\int_a^\infty \frac{dq}{\psi(q)}<\iy \ \Leftrightarrow \ \int_a^\infty \frac{dq}{\psi_{\delta}^r(q)}<\iy.
\end{equation*}
It is perhaps not a priori clear why $\Pi_\delta^r$ is a simpler process.
We will soon see that 
because
its $\X$-measure is concentrated on $\Delta^\delta$, 
the criterion of \cite{schweinsberg_xi} for CDI is sharp, and
under an additional condition, its asymptotic speed can be found 
in a way analogous to \cite{bbl1}.

The second variation is as follows:
given realizations of (\ref{DPPPpi}) and (\ref{Dcoloring}) as before,
define the {\em color-joining} 
$\Pi^j$ ($j$ stands for ``joining'') of $\Pi$ to be the 
partition valued process where all the blocks with integral color are 
immediately merged together into one block.
As for the $\delta$-reduction,
one can obtain a coupling of $\Pi^{(n),\eps}$ (resp.~$\Pi$) 
and its color-joining
$\Pi^{(n),\eps,j}(t)$ (resp.~$\Pi^j$), so that
there are fewer blocks contained in 
$\Pi^{(n),\eps,j}(t)$
 (resp.~$\Pi^j$)
than in $\Pi^{(n),\eps}$ (resp.~$\Pi$) at all times.

The coalescent $\Pi^j$ should be a $\La$-coalescent,
with its corresponding $\psi_\La$ from (\ref{D:psi la}) given by
\begin{equation}
\label{Epsi join}
\psi^j(q) = 
\int_{\Delta}\frac{ (e^{-q\sum_{i=1}^\infty x_i}-1+q\sum_{i=1}^\infty x_i) }
{\sum_{i=1}^\infty x_i^2}\, \X (d\bx). 
\end{equation}
The slick point is that the right-hand side in (\ref{Epsi join}) may be infinite.
The existence of the integral in (\ref{Epsi join}) is equivalent to (cf.~the condition
(\ref{ERcond}) in the next section)
\begin{equation*}
\label{Epsi auxi}
\int_{\Delta} \frac{(\sum_{i=1}^\infty x_i)^2}{\sum_{i=1}^\infty x_i^2}\, \X (d\bx) < \infty.
\end{equation*}
Indeed, we have
\begin{equation}
\label{Earriveto}
\left(\sum_{i=1}^\infty x_i\right)^2\frac{q^2 \wedge q }{10}  
\leq e^{-q\sum_{i=1}^\infty x_i}-1+q\sum_{i=1}^\infty x_i \leq  \left(\sum_{i=1}^\infty x_i\right)^2\frac{q^2}{2}.
\end{equation}
The upper bound is just an application of (\ref{Ecalcfact}).
For the lower bound, assume that $q\geq 1$, the argument is simpler otherwise.
Note that
$e^{-z}-1 +z \geq z^2/2 -z^3/3 \geq z^2/10$ for $z< 5/4$,  and that
$e^{-z}-1 +z \geq z/10$ for $z\geq 5/4$.
Substituting $z=q \sum_{i=1}^\infty x_i$ 
%and using $z^2\wedge z \geq q (\sum_{i=1}^\infty x_i)^2$ 
we arrive at (\ref{Earriveto}).
As a consequence, the right-hand-side in (\ref{Epsi join}) is finite for one $q\in \R_+$ if and only 
if it is finite for all $q \in \R_+$.

\section{Main results}
\label{S:results}
\subsection{Regular case}
In this subsection assume that
(in addition to $\X(\{0\}\cup \Delta_f)=0$) the measure $\X$ satisfies the 
{\em regularity} condition 
\begin{equation}
\label{ERcond}
\tag{R}
\int_\Delta \frac{(\sum_{i=1}^\iy x_i)^2}{\sumd} \,\X(d\bx) < \iy.
\end{equation}
The complementary setting is discussed in Section \ref{S:improper joint}.

Denote by
$(N^\X(t), t\geq 0)$ the number of blocks process for the
$\X$-coalescent $(\Pi(t),t\geq 0)$, and recall definition (\ref{Ev}).
%Assume that the $\La$-coalescent comes down from infinity.
Regularity (\ref{ERcond}) implies
\[
\int_{\Delta \setminus \Delta^{1-a}} \frac{1}{\sumd} \,\X(d\bx) < \iy,
\]
for any $a \in (0,1)$.
In particular, in the PPP construction an atom $(t,\bx)$ satisfying
$\sum_i x_i > a$ appears at a strictly positive random (exponential) time.
Therefore, if $\X^a(d\bx)=\X(d\bx)\indica{\sum_i x_i \leq a}$,
the $\X$-coalescent and the $\X^a$-coalescent
 have the same small time behavior.

As already indicated in Section \ref{S:main preview}, the central result of this paper is
\begin{theorem} \label{Tsmalltime reg}
If both $\X(\{0\}\cup \Delta_f)=0$ and (\ref{ERcond}) hold, then
\begin{equation*}
\label{nblocks}
\lim_{t\to 0}\frac{N^\X(t)}{v_\X(t)} = 1, \mbox{ almost surely},
\end{equation*}
where $\iy/\iy \equiv 1$.
In particular, under these assumptions,
the candidate speed is finite if and only if
the $\X$-coalescent comes down from infinity, which happens if and only if it
 is the true speed of CDI.
\end{theorem}
The proof is postponed until Section \ref{S:martingale}.
The importance of condition (\ref{ERcond}) will become evident 
in view of Lemma \ref{Lbinlogcalc}, that implies Proposition \ref{martingale estimates} 
(see also~(\ref{eqdriftvar})--(\ref{eqdriftvar2})),
which is an
essential ingredient in the martingale analysis of 
Section \ref{S:martingale}.
It is interesting to note that (\ref{ERcond}) arises independently in the
context of the color-joining construction in Section \ref{S:two op}.
\begin{remark}
Once given Theorem \ref{Tsmalltime reg}, by straightforward copying of arguments from \cite{bbl1}, one could 
obtain the convergence of $N^\X/v_\X$ 
in the $L^p$ sense for $p\geq 1$, as well as the convergence of 
the total length of the genealogical tree in the regular setting.
\end{remark}
\begin{lemma}
\label{Ljoining reg}
Under (\ref{ERcond}) the color-joining $\Pi^j$ is a $\La$-coalescent corresponding to 
$\psi^j$ from (\ref{Epsi join}).
If $\Pi$ comes down from infinity, then $\Pi^j$ comes down from infinity at least as 
fast as $\Pi$, meaning that $\Pi^j$ has the speed of CDI
$v_\X^j(t)$ determined by
\[
\int_{v_\X^j(t)}^\infty \frac{dq}{\psi^j(q)} =t, 
\] 
where $v_\X^j(t)\leq v_\X(t)$, for any $t> 0$.
\end{lemma}
\begin{proof}
 Consider the process $\Pi^{(n)}=\lim_{\eps \to 0} \widetilde{\Pi}^{(n),\eps}$
from the PPP construction of $\Pi$.
It suffices to show that $\Pi^{(n),j}$ is a $\La$-coalescent (started from a configuration of $n$ blocks) corresponding to 
\[
\La(dy)=y^2 \int_{\Delta \cap \{\sum_i x_i=y\}} \frac{\X(d\bx)}{\sumd}
\]
This is an immediate consequence of elementary properties of the Poisson point process $\pi$ from
(\ref{DPPPpi}).

From the coupling of $\Pi$ and $\Pi^j$, 
where $\Pi$ has at least as many blocks as $\Pi^j$ at any positive time,
it is clear that if $\Pi$ comes down from infinity,
then also does $\Pi^j$. 
Alternatively, the reader can verify analytically that
\[
(\psi^j(q)-\psi_\X(q))' \geq 0, \ q \geq 0.
\]
Since $\Pi^j$ is a $\La$-coalescent, we know that $v_\X^j$ is its speed of CDI.
Then again due to the above coupling of $\Pi$ and $\Pi^j$ 
we conclude that $v_\X^j(t)\leq v_\X(t)$, for any $t> 0$.
\end{proof}
\begin{remark}
Note that all the $\X$-coalescents with $\X$ of the form (\ref{ELa from Xi}) are regular, and more generally, if $\X$ is supported on any ``finite''
subsimplex $\{\bx:x_k=0, \, \forall \, k\geq n\}$ of $\Delta$, then
the corresponding $\X$-coalescent is regular.
In particular, the $\X$-coalescents featuring in the selective sweep approximation of \cite{dursch,schdur}
are regular.
\end{remark}

\subsection{Non-regular case}
\label{S:improper joint}
Assume $\X(\{0\}\cup \Delta_f)=0$ as in the previous subsection.
The setting where
\begin{equation}
\label{EcondNR}
\tag{NR}
\int_\Delta \frac{(\sum_{i=1}^\iy x_i)^2}{\sumd} \,\X(d\bx) = \iy
\end{equation}
is more complicated, and the small time 
asymptotics for such $\X$-coalescents is only partially understood.

Due to observations made in the previous section,
(\ref{EcondNR}) is equivalent to the fact that the integral in
 (\ref{Epsi join}) diverges.
\begin{lemma}
\label{Ljoining}
Under (\ref{EcondNR}) the color-joining $\Pi^j$ is a trivial process containing
one block at all positive times. 
\end{lemma}
\begin{proof}
As for Lemma \ref{Ljoining reg}, consider the prelimit coalescents $\Pi^{(n)}$
and their color-joinings $\Pi^{(n),j}$.
It is easy to verify that (\ref{EcondNR}) implies instantaneous coalescence
of any two blocks of  
$\Pi^{(n),j}$, almost surely.
Indeed, the rate of coalescence for a pair of blocks is given by the integral in (\ref{EcondNR}).
\end{proof}
\begin{remark}
The last lemma holds even if
$\Pi$ does not come down from infinity.
\end{remark}
The following illuminating example was given in \cite{schweinsberg_xi}.
Suppose $\X$ has an atom of mass $1/2^n$ at 
\[
\bx^n:=(x_1^n,\ldots, x_{2^n-1}^n,0,\ldots),
\]
where $x_i^n = 1/2^n$, $i=1,\ldots,2^n-1$, and $n\in \N$.
Then $\psi_\X(q)=\Theta(q\log(q))$ so $v_\X$ is infinite,
but the corresponding $\Pi$ comes down from infinity.
Due to Theorem \ref{Tsmalltime reg} we see that (\ref{ERcond}) cannot hold in this case. 

It is useful to consider a generalization as follows: for a 
sequence $f:\N\to (0,1)$, let $\X$ have atom of mass $1/2^n$ at 
\[
\bx^n:=(x_1^n,\ldots, x_{\lfloor f(n)2^n\rfloor }^n,0,\ldots),
\]
where again $x_i^n = 1/2^n$, $i=1,\ldots,\lfloor f(n)2^n\rfloor$, and where we assume that
$\lfloor f(n)2^n\rfloor \in \{1,\ldots,2^{n-1}-1\}$, $n\geq 1$, so that
$\X(\Delta_f)=0$.
It turns out that again $\psi_\X(q)=\Theta(q\log(q))$ (in fact, this asymptotic behavior 
is uniform in the above choice of $f$),
while the integral in (\ref{ERcond}) (or (\ref{EcondNR})) is asymptotic to 
\[
\sum_n f(n).
\]
Due to Theorem \ref{Tsmalltime reg} we see that as soon as the above series converges, the 
corresponding $\X$-coalescent does not come down from infinity.
However, its color-joining will in many cases come down from infinity, 
for example if $f(n)= n^{-2}$, then 
$\psi^j(q) = \Theta(q^{3/2})$.

\begin{proposition}
Suppose that (\ref{EcondNR}) holds and that the 
the corresponding (standard) $\X$-coalescent $\Pi$ has an infinite candidate speed 
(or equivalently, that (\ref{E me cond}) fails).
Then for any $\delta \in (0,1)$, its $\delta$-reduction
$\Pi_\delta^r$ does not come down from infinity.
\end{proposition}
\begin{proof}
If $\Pi$ does not come down from infinity, then 
$\Pi_\delta^r$ does not either, due to the monotone coupling of $\Pi$ and
$\Pi_\delta^r$.

Even if $\Pi$ comes down from infinity, we have that 
\[
\int_{\Delta \setminus \Delta^\delta} \frac{1}{\sumd}\, \X_\delta^r(d\bx)=0 < \infty,
\]
and, as already observed, that
$\int_a^\infty 1/\psi_\delta^r(q)\,dq =\infty$ (for one and then all $a\in (0,\infty)$),
so due to Proposition 33 of \cite{schweinsberg_xi}, 
$\Pi_\delta^r$ does not come down from infinity.
\end{proof}
The last result and Lemma \ref{Ljoining} indicate the level of opacity
of the non-regular setting.
Indeed, a $\X$-coalescent $\Pi$ 
that comes down from infinity, but has infinite candidate speed and satisfies (\ref{EcondNR}),
can be formally ``sandwiched'' between its 
corresponding $\Pi^j$ and $\Pi_\delta^r$, where $\delta>0$ is very small, however
the lower bound $\Pi^j$ is trivial, and the upper bound $\Pi_\delta^r$ 
does not come down from infinity,
so one gains no pertinent information from the coupling.

To end this discussion, let us mention another class of frustrating examples.
Suppose that $\X_1$ is a probability measure on $\Delta$ satisfying 
both (\ref{ERcond}) and (\ref{E me cond}) 
and denote by $v_1$ the speed of CDI for the corresponding $\X_1$-coalescent.
Let $\X_2$ be a probability measure on $\Delta$ satisfying (\ref{EcondNR}).
Define
\[
\X:= \frac{1}{2}(\X_1+\X_2),
\] 
so that $\X$ satisfies both (\ref{E me cond}) and (\ref{EcondNR}).
Due to easy coupling, the $\X$-coalescent comes down from infinity, and moreover
\[
\limsup_{t\to 0} \frac{2N^\X(t)}{v_1(t)}=1.
\]
The martingale technique however breaks in the
non-regular setting, and we have no further information on the small time
asymptotics of $N^\X$.
It seems reasonable to guess that $N^\X$ is asymptotic to $v_1/2$ as $t\to 0$.

Remark \ref{R:last} discusses an approach that might 
be helpful in resolving the question of speed for $\X$-coalescents that come down from infinity
in the non-regular setting.

\section{The arguments}
\label{S:martingale}
The goal of this section is to prove Theorem \ref{Tsmalltime reg}
(that is, Theorem \ref{Tsmalltime reg early})
 by adapting the technique
from \cite{bbl1}.

As already noted, the function 
$\psi$ defined in (\ref{D:psi xi}) is
strictly increasing and convex.
Furthermore, it is easy to check that
$v'(s)=-\psi(v(s))$ where $v=v_\X$ is defined in (\ref{Ev}),
so that both $v$ and $|v'|$ are decreasing
functions.

Due to the observation preceding the statement of
Theorem \ref{Tsmalltime reg},
we can suppose without loss of generality that ${\rm supp}(\X)
\subset \Delta^{3/4}$ (recall notation (\ref{EDeleps})).
As in \cite{bbl1}, this will simplify certain technical estimates.

To shorten notation, write $N$ instead of $N^\X$.
Note that the function $v$ is the unique solution of the following
integral equation
\begin{equation}\label{IE for v}
\log(v(t)) - \log(v(z)) +  \int_z^t \frac{\psi(v(r))}{v(r)}\, dr = 0,
\ \forall 0<z<t,
\end{equation}
with the ``initial condition'' $v(0+)=\infty$.
If $\X$ can be identified with a probability measure $\La$ on $[0,1]$
as in (\ref{ELa from Xi}), then (\ref{IE for v}) is identical to the 
starting observation in the proof 
of Theorem \ref{Tsmalltime reg} for $\La$-coalescents
(cf.~proof of \cite{bbl1} Theorem 1).

Indeed, the rest of the argument is analogous to the 
one from \cite{bbl1} for $\La$-coalescents,
the general regular $\X$-coalescent setting 
being only slightly more complicated.
The few points of difference will be treated in detail, while the 
rest of the argument is only sketched.

\subsection{Preliminary calculations}
Assume that the given $\X$-coalescent has a finite number of blocks at some positive time $z$.   
Consider the process
\begin{equation*}
\label{EMart almost} 
M(t):=\log(N(t)) - \log(N(z)) + \int_z^t
\frac{\psi(N(r))}{N(r)} dr, \ t \geq z.
\end{equation*}
Let $n_0\ge 1$ be fixed. Define
\begin{equation}
\label{E taun0}
\tau_{n_0} :=\inf\{s>0: N(s) \leq n_0\}.
\end{equation}

It turns out that, under the regularity hypothesis (\ref{ERcond}),
 $M(t\wedge \tau_{n_0})$ is
``almost'' (up to a bounded drift correction)
a local martingale, with respect to
the natural filtration $(\FF_t,t\ge 0)$
generated by the underlying $\X$-coalescent
process.
\begin{proposition} \label{martingale estimates}
There exists some deterministic $n_0\in \N$ and $C<\infty$ such that
\begin{equation}
E[d\log(N(s))|\FF_s]= \left(-\frac{\psi(N(s))}{N(s)} + h(s)\right) ds,
\label{Eha}
\end{equation}
where $(h(s), s\ge z)$ is an $\FF$-adapted process  such that
$\sup_{s\in [z,z \wedge \tau_{n_0}]} |h(s)|\le C$, and
\[
E[[d\log(N(s))]^2|\FF_s] \indic{\{s \leq \tau_{n_0}\}} \leq C\, ds, \mbox{ almost surely.}
\]
Both estimates are valid uniformly over $z>0$.
\end{proposition}
Restricting the analysis to $n$ larger than $n_0$ is a
consequence
of the following estimate,
whose proof is given immediately after the proof of the proposition.
Recall $Y_\ell^{(n)}$ defined in (\ref{EYs}).
When taking probabilities or expectations with respect
to the joint law
of $(Y_\ell^{(n)},\,\ell\geq 1)$, 
we include subscript $\bx$ to indicate the dependence
of the law on $\bx$. 
Define
\[
S(\bx):= \sumd +\left(\sum_{i=1}^\infty x_i\right)^2.
\]
\begin{lemma}
\label{Lbinlogcalc} There exists $n_0\in \N$ and $C_0<\infty$ such
that for all $n\geq n_0$ and all $\bx \in \Delta^{3/4}$, we have
\[
\left| E_{\bx}\!\left(\!\! \log\!\!\left[n - \sum_{\ell=1}^\infty (Y_\ell^{(n)} -\indica{Y_\ell^{(n)} >0})\right]\!\! - \log{n}\!\!\right)\! + \!\frac{\sum_{\ell=1}^\infty  nx_\ell-1+(1-x_\ell)^n}{n}
\right| \leq C_0 S(\bx),
\]
and
\[
E_{\bx}\left( \log\!\!\left[n - \sum_{\ell=1}^\infty (Y_\ell^{(n)} -\indica{Y_\ell^{(n)} >0})\right]\!\! - \log{n}\right)^2 \leq C_0 S(\bx).
\]
\end{lemma}
\begin{proof}
{\em $[$of Proposition \ref{martingale estimates}$]$}
Since (\ref{ERcond}) holds, it suffices to
show that for each $s>0$, we have on $\{N(s) \geq n_0\}$
\begin{equation}\label{eqdriftvar}
\left|\frac{E(d \log(N(s))|\FF_s)}{ds} + \frac{\psi(N(s))}{N(s)} \right|=
|h(s)|= O\left(\int_{\Delta^{3/4}}\frac{S(\bx)}{\sumd}\, \X(d\bx)\right),
\end{equation}
and
\begin{equation}\label{eqdriftvar2}
E([d \log(N(s))]^2|\FF_s) =O\left(\int_{\Delta^{3/4}} \frac{S(\bx)}{\sumd}\, \X(d\bx)\right)ds,
\end{equation}
where $O(\cdot)$ can be taken uniformly in $s$.
Note that the finite integrals above are in fact taken over $\Delta$,
since $\X$ is supported on $\Delta^{3/4}$.

Recall the PPP construction of
Section \ref{S:Xicoal further} and fix $n\geq n_0$. 
On the event $\{N(s)=n\}$, 
an atom carrying value $\bx\in\Delta$ arrives at rate $1/(\sumd)\,\X(d\bx)\,ds$, and given its arrival, $\log{N(s)}=\log{n}$ jumps to 
$\log(n - \sum_{\ell=1}^\infty (Y_\ell^{(n)} -\indica{Y_\ell^{(n)} >0}) )$. Therefore,
\[
E(d \log(N(s))|\FF_s)=\int_{\Delta} E_{\bx}\!\!\left[\log\frac{n-\sum_{\ell=1}^\infty (Y_\ell^{(n)} -\indica{Y_\ell^{(n)} >0})}{n} \right]\!\!\frac1{\sumd}\,\X(d\bx) \,ds.
\]
Due to Lemma \ref{Lbinlogcalc} and (\ref{E easy bound})
we can now derive (\ref{eqdriftvar}).

To bound the infinitesimal variance on the event $\{N(s) =n\}$,
use the second estimate in Lemma \ref{Lbinlogcalc}, together with the fact
\[
\frac{E([d \log(N(s))]^2|\FF_s)}{ds}\leq
\int_{\Delta} E_{\bx}\!\!\left[\log^2\!\!\left(\!\!\frac{n-\sum_{\ell=1}^\infty (Y_\ell^{(n)} -
\indica{Y_\ell^{(n)} >0})}{n}\right)\!\right]\!\!
\frac1{\sumd}\,\X(d\bx).
\]
Finally, note that both (\ref{eqdriftvar}) and (\ref{eqdriftvar2}) are uniform upper bounds over 
$s$.
\end{proof}

\begin{proof}
{\em $[$of Lemma \ref{Lbinlogcalc}$]$}
The argument is almost the same as that for \cite{bbl1} Lemma 19 in the $\La$-coalescent setting.
Since the regularity ``dichotomy'' is a consequence of some
more complicated expressions (arising in the calculations) in the current setting, 
most of the steps are included. 
Abbreviate 
\[
Z^{(n)}:=\frac{\sum_{\ell=1}^\infty (Y_\ell^{(n)} -
\indica{Y_\ell^{(n)} >0})}{n},
\]
and note that 
%due to the assumptions on the support of $\X$,
$Z^{(n)}$ is stochastically bounded by a Binomial($n,\sum_i x_i$) random
variable.
Let
\begin{equation*}
\label{ET} T\equiv T_n := \log\left(1-Z^{(n)}\right).
\end{equation*}
Split the computation according to the event
\[
A_n =\{Z^{(n)} \leq 1/2\},
\]
whose complement 
has probability bounded by
\begin{equation*}
\label{ELDPbound} \exp\left\{ -n
\left(\frac{1}{2}\log{\frac{1}{2p}}+
\frac{1}{2}\log{\frac{1}{2(1-p)}} \right)\right\} = 2^n p^{n/2}
(1-p)^{n/2},
\end{equation*}
uniformly in $p:=\sum_i x_i\leq 1/4$ and $n$, due to a large deviation bound (for sums of i.i.d.~Bernoulli random variables).
%\begin{equation}
%\label{Ecalcfact}
%\end{equation}
On $A_n^c$ we have $|T| \leq \log{n}$, and on $A_n$ we apply a
calculus fact,
%\begin{equation}
$|\log(1-y) + y| \leq \frac{y^2}{2(1-y)}\leq y^2, \ y\in [0,1/2]$,
%\label{Ecalc_fact}
%\end{equation}
to obtain
\[
\left|E[T] + E\left[Z^{(n)} \indic{A_n}\right]
\right| \leq (\log{n}) P(A_n^c) +
E\left[(Z^{(n)})^2 \indic{A_n}\right].
\]
Since $Z^{(n)}\leq 1$, we conclude
\begin{equation*}
\label{Ehelping}
\left|E[T] + E[Z^{(n)}] \right| \leq
(\log{n}+1) P(A_n^c) + E[(Z^{(n)})^2].
\end{equation*}
Note that $|E[T]+E[Z^{(n)}]|$ is precisely the left-hand side of
the first estimate stated in the lemma. 
Due to the estimate (52) in the proof of \cite{bbl1} Lemma 19 we have
\begin{equation*}
\label{EPlognbound} 
(\log{n}) P(A_n^c)\leq (\log{n}) 2^n p^{n/2}
(1-p)^{n/2}\leq C p^2< CS(\bx),
\end{equation*}
for some $C<\infty$, all $p\in [0,1/4]$, and all $n$ large. 

Until this point the argument is identical to the one
for $\La$-coalescents.
The new step is verifying that 
\begin{equation}
\label{E new bound}
E[(Z^{(n)})^2]\leq S(\bx).
\end{equation}
It is easy to check (see for example \cite{bbl1} Corollary 18)
that
\begin{equation}
\label{Evariancebd}
E[(Y_\ell^{(n)}-\indica{Y_\ell^{(n)}>0})^2] \leq C n^2 (x_\ell)^2,
\end{equation}
for some constant $C<\infty$.
For two different indices $k,\ell$,
use Cauchy-Schwartz inequality together with the above bound to get 
\begin{equation}
\label{Ecovariancebd}
|E[(Y_k^{(n)}-\indica{Y_k^{(n)}>0})(Y_\ell^{(n)}-\indica{Y_\ell^{(n)}>0})]| \leq \sqrt{C^2 n^4 (x_k)^2 (x_\ell)^2}=
C n^2 x_k x_\ell.
\end{equation}
One obtains (\ref{E new bound}) from (\ref{Evariancebd})--(\ref{Ecovariancebd}) after rewriting
$E[(Z^{(n)})^2]$ as
\[
\frac{1}{n^2} \left(\sum_\ell E[(Y_\ell^{(n)}-\indica{Y_\ell^{(n)}>0})^2] + 
\sum_k\sum_{\ell \neq k} E[(Y_\ell^{(n)}-Y_\ell^{(n)}>0) (Y_k^{(n)}- \indica{Y_k^{(n)}>0})]\right).
\]

The second estimate is proved exactly as in \cite{bbl1}.
\end{proof}
\begin{remark}
The expectation of the product of
 $Y_k^{(n)}-\indica{Y_k^{(n)}>0}$ and $Y_\ell^{(n)}-\indica{Y_\ell^{(n)}>0}$ can be
computed explicitly, and one can verify that its absolute value has the
order of magnitude $n^2 x_k x_\ell$ as $x_k$ and (or) $x_\ell$ tend to $0$.
\end{remark}

\subsection{Proof of Theorem \ref{Tsmalltime reg}}
\label{S:smalltime}

{\em Part I.}
Suppose that a given regular $\X$-coalescent starts from $n$ blocks, where $n\in \N$ is large 
and finite. In other words, consider the prelimit process $\Pi^{(n)}$. 

Recall Remark \ref{R cdi lam}.
Define a family of deterministic functions $(v^n,\,n\in \N)$ 
as in (\ref{Evns}), where $\psi=\psi_\X$,
and note that $v^n$ satisfies $v^n(0)=n$ and
\begin{equation}\label{IE for v shift n}
\log(v^n(t)) - \log(n) +  \int_0^t \frac{\psi(v^n(r))}{v^n(r)}\, dr = 0,
\ \forall t>0.
\end{equation}
It is easy to see that the following is true.
\begin{lemma}
\label{L conv v}
We have $v^n(t) \leq v^{n+1}(t)$ and $\lim_{n\to \infty} v^n(t) = v_\X(t)$, for each 
$t>0$.
\end{lemma}
For each 
$n\geq n_0$ (where $n_0$ is the parameter from Proposition \ref{martingale estimates}) define
the process
\[
M_n(t):=
\log{\frac{N^n(t\wedge \tau_{n_0}^n)}{v^n(t \wedge \tau_{n_0}^n)}} 
+
\int_0^{t\wedge \tau_{n_0}^n}
\left[
\frac{\psi(N^n(r))}{N^n(r)} -
\frac{\psi(v^n(r))}{v^n(r)}
+h(r)\right] dr, \ t \ge 0,
\]
where $h=h^n$ is given in (\ref{Eha}), and $\tau_{n_0}^n:=\inf\{s>0: N^n(s) \leq n_0\}$ in 
analogy to (\ref{E taun0}).

Due to Proposition \ref{martingale estimates} and (\ref{IE for v shift n}),
we know that $M_n$ is a martingale (note that $M_n(0)= 0$), such that
\[
E[(M_n(s)- M_n(u))^2|\FF_s] \leq C(s-u),
\]
uniformly over $n\geq n_0$ and $u,s$ such that $s\geq u\geq 0$.
Fix any $\alpha \in (0,1/2)$. 
Doob's $L^2$-inequality therefore implies
\begin{equation}
\label{EboundDoob n}
P(\sup_{t\in[0,s]} |M_n(t)| > s^{\alpha})= O(s^{1-2\alpha}),
\end{equation}
where $O(\cdot)$ term is uniform over $n\geq n_0$.
Due to Proposition \ref{martingale estimates}, the term
\[ 
\int_0^{t\wedge \tau_{n_0}^n} h(r) dr
\]
is of smaller order $O(s)$, again uniformly in $n\geq n_0$. 
Hence we obtain from (\ref{EboundDoob n}) that
\[
P\left(\sup_{t\in[0,s]} \left|
\log{\frac{N^n(t\wedge \tau_{n_0}^n)}{v^n(t \wedge \tau_{n_0}^n)}} +
\int_0^{t\wedge \tau_{n_0}^n}
\left[\frac{\psi(N^n(r))}{N^n(r)} -\frac{\psi(v^n(r))}{v^n(r)}\right] dr \right|
> s^{\alpha}
\right) = O(s^{1-2\alpha}).
\]
Due to Lemma \ref{Lincreasing} and \cite{bbl1} Lemma 10, the last estimate implies in turn
\begin{equation}
\label{Esuffices n}
P\left(\sup_{t\in[0,s]} \left|
\log{\frac{N^n(t\wedge \tau_{n_0}^n)}{v^n(t \wedge \tau_{n_0}^n)}} \right|
> 2s^{\alpha}
\right) = O(s^{1-2\alpha}).
\end{equation}

Assume that the corresponding regular (standard) $\X$-coalescent $\Pi$ comes down from infinity.
Since $N^n(t)\nearrow N(t)$, for each $t >0$, $\tau_{n_0}^n \nearrow \tau_{n_0}$, 
almost surely, and since
\begin{equation}
\label{Etaunzero}
P(\tau_{n_0} >0)=1,
\end{equation}
we obtain due to Lemma \ref{L conv v} that the candidate speed $v(t):=\lim_n v^n(t)$ is finite
for each $t>0$.

Conversely, if this $\X$-coalescent does not come down from infinity,
then it must be $v(t):=\lim_n v^n(t)=\infty$.

{\em Part II.} 
Suppose that the $\X$-coalescent from part I comes down from infinity.
It is tempting to let $n\to \infty$ in (\ref{Esuffices n})
in order to obtain
\begin{equation}
\label{Esuffices x}
P\left(\sup_{t\in[0,s]} \left|
\log{\frac{N(t\wedge \tau_{n_0})}{v(t \wedge \tau_{n_0})}} \right|
> 2s^{\alpha}
\right) = O(s^{1-2\alpha}).
\end{equation}
However, this step would not be rigorous without additional information
on the family of events in (\ref{Esuffices n}), indexed by $n\geq n_0$.
An alternative approach is discussed next. 

From part I we know that the corresponding candidate speed
is finite.
Using this fact, a variation of the previous argument 
yields (\ref{Esuffices x}). 
Define a family of deterministic functions $(v_x,\,x\in \R)$ by
\[
v_x(t)= v(t+x),\ t\geq -x,
\]
and note that each $v_x$ satisfies an appropriate analogue of (\ref{IE for v}) on its entire domain,
more precisely, $ v_x(-x+)=\infty$ and
\begin{equation}\label{IE for v shift}
\log(v_x(t)) - \log(v_x(z)) +  \int_z^t \frac{\psi(v_x(r))}{v_x(r)}\, dr = 0,
\ \forall -x<z<t.
\end{equation}
Due to (\ref{Etaunzero}), one can assume that $z\geq \tau_{n_0}$.
For each 
%fixed $z>0$ and each 
$x>-z$ define
\[
M_{z,x}(t):=
\log{\frac{N(t\wedge \tau_{n_0})}{v_x(t \wedge \tau_{n_0})}} -
\log{\frac{N(z)}{v_x(z)}}+
\int_z^{t\wedge \tau_{n_0}}
\left[
\frac{\psi(N(r))}{N(r)} -
\frac{\psi(v_x(r))}{v_x(r)}
+h(r)\right] dr, \ t \ge z,
\]
where $h$ is given in (\ref{Eha}).

It will be convenient to consider for each fixed $z>0$ a process $M_{z,X}$, where
$X\in \FF_z$ such that $P(X>-z)=1$.
Note that such $M_{z,X}$ is adapted to the filtration
$(\FF_r,\,r\geq z)$.
More precisely, let $X_z$ be the random variable defined by
\begin{equation*}
\label{def Xz}
N(z)=v(X_z+z)= v_{X_z}(z).
\end{equation*}
It is easy to see that $X_z +z$ is decreasing to $0$ as $z$ decreases to $0$, and that
therefore the following is true.
\begin{lemma}
\label{L conv Xz}
We have $\lim_{z\to 0} X_z = 0$, hence $\lim_{z\to 0} v_{X_z}(t) = v(t)$ for all 
$t>0$, almost surely.
\end{lemma}
Due to Proposition \ref{martingale estimates} and (\ref{IE for v shift}),
we know that $M_{z,X_z}$ is a martingale (note that $M_{z,X_z}(0)= 0$), such that
\[
E[(M_{z,X_z}(s)- M_{z,X_z}(u))^2|\FF_s] \leq C(s-u),
\]
uniformly over $u,s$ such that $s\geq u\geq z$.
%Fix any $\alpha \in (0,1/2)$. Due to Doob's $L^2$-inequality, we
%therefore have
As in part I, we obtain
\begin{equation*}
\label{EboundDoob}
P(\sup_{t\in[z,s]} |M_{z,X_z}(t)| > s^{\alpha})= O(s^{1-2\alpha}),
\end{equation*}
where $O(\cdot)$ term is uniform over $z>0$.
Again, due to Proposition \ref{martingale estimates}, the term
\[ 
\int_z^{t\wedge \tau_{n_0}} h(r) dr
\]
is of smaller order $O(s)$, uniformly in $z$. 
Hence 
\[
P\left(\sup_{t\in[z,s]} \left|
\log{\frac{N(t\wedge \tau_{n_0})}{v_{X_z}(t \wedge \tau_{n_0})}} +
\int_z^{t\wedge \tau_{n_0}}
\left[\frac{\psi(N(r))}{N(r)} -\frac{\psi(v_{X_z}(r))}{v_{X_z}(r)}\right] dr \right|
> s^{\alpha}
\right) = O(s^{1-2\alpha}).
\]
As before, due to Lemma \ref{Lincreasing} and \cite{bbl1} Lemma 10, the last estimate implies 
\[
P\left(\sup_{t\in[z,s]} \left|
\log{\frac{N(t\wedge \tau_{n_0})}{v_{X_z}(t \wedge \tau_{n_0})}} \right|
> 2s^{\alpha}
\right) = O(s^{1-2\alpha}), 
\]
and therefore for any $z'<z$
\[
P\left(\sup_{t\in[z,s]} \left|
\log{\frac{N(t\wedge \tau_{n_0})}{v_{X_{z'}}(t \wedge \tau_{n_0})}} \right|
> 2s^{\alpha}
\right) = O(s^{1-2\alpha}).
\]
Let $z'\to 0$ and use Lemma \ref{L conv Xz}, and then let $z\to 0$ to obtain
(\ref{Esuffices x}).
This together with (\ref{Etaunzero}) shows that in this setting
the candidate speed is 
the true speed of CDI.
\begin{remark}
\label{R:last}
As already mentioned, the above argument works only under the assumption (\ref{ERcond}).
However, regularity is only needed in 
linking $E(d\log{N(t)}|\FF_t)$ to $-\psi(N(t))/N(t)$, and in uniformly
bounding the infinitesimal variance of $\log{N(t)}$.
For irregular $\X$-coalescents that 
have an infinite candidate speed, but also
come down from infinity, 
a relation of similar kind
\[
E(d\log{N(t)}|\FF_t)=-\frac{\psi_1(N(t))}{N(t)} + h(t)
\] 
might be possible, where $h$ is still a uniformly bounded process, and where
$\psi_1$ is an increasing, convex function satisfying Lemma \ref{Lincreasing} and
\[
\int_a \frac{dq}{\psi_1(q)} <\infty, \ a>0.
\]
It is natural to guess that $v_1:\R_+\to \R_+$, determined by $\int_{v_1(t)}^\infty dq/\psi_1(q)=t$,
is then the speed of CDI.
\end{remark}

{\bf Acknowledgement.}
The author wishes to thank the staff at the Institut Mittag-Leffler for their hospitality.

\end{document}